\documentclass[11pt,twoside]{article}
\usepackage{amsmath, amsthm, graphicx}
\usepackage{amsfonts}
\usepackage[english]{babel}
\usepackage{float}

\newtheorem{theorem}{Theorem}[section]

\newtheorem{lemma}[theorem]{Lemma}
\newtheorem{corollary}[theorem]{Corollary}

\date{}
\begin{document}

\title{{\Large\bf Maps preserving the
local spectral subspace of skew-product of operators}}

\author{{\normalsize\sc R. Parvinianzadeh \footnote{ Corresponding author: }}\\[0.5cm]
 Department of Mathematics, University of Yasouj,
 Yasouj 75918, Iran\\
 \footnote{E-mail addresses: r.parvinian@yu.ac.ir.}
 }
\maketitle

{\footnotesize  {\bf Abstract} Let $B(H)$ be the algebra of all bounded linear operators on an infinite-dimensional complex Hilbert space $H$. For $T \in B(H)$ and $\lambda \in \mathbb{C}$, let $H_{T}(\{\lambda\})$ denotes the local spectral subspace of $T$ associated with $\{\lambda\}$.
We prove that if $\varphi:B(H)\rightarrow B(H)$ be an additive map such that its range contains all operators of rank at most
two and satisfies
$$H_{\varphi(T)\varphi(S)^{\ast}}(\{\lambda\})= H_{TS^{\ast}}(\{\lambda\})$$
for all $T, S \in B(H)$ and $\lambda \in \mathbb{C}$, then there exist a unitary operator $V$ in $B(H)$ and a nonzero scalar $\mu$ such that $\varphi(T) = \mu TV^{\ast}$ for all $T \in B(H)$.
We also show if $\varphi_{1}$ and $\varphi_{2}$ be additive maps from $B(H)$ into $B(H)$ such that their ranges contain all operators of rank at most
two and satisfies
$$H_{\varphi_{1}(T)\varphi_{2}(S)^{\ast}}(\{\lambda\})= H_{TS^{\ast}}(\{\lambda\})$$
for all $T, S \in B(H)$ and $\lambda \in \mathbb{C}$. Then $\varphi_{2}(I)^{\ast}$ is invertible, and $\varphi_{1}(T) = T(\varphi_{2}(I)^{\ast})^{-1}$ and $\varphi_{2}(T) =\varphi_{2}(I)^{\ast}T$ for all $T \in B(H)$.\\

\noindent {\bf Mathematics Subject Classification:} Primary 47B11: Secondary 47A15, 47B48.\\

\noindent {\bf Keywords}: Local spectrum, Local spectral subspace, Nonlinear preservers, Rank-one operators.}

\section{\normalsize\bf Introduction}
Throughout this paper, $H$ and $K$ are infinite-dimensional complex Hilbert
spaces. As usual $B(H,K)$ denotes the space of all bounded linear operators from
$H$ into $K$. When $H=K$ we simply write $B(H)$ instead of $B(H,H)$, and its unit will be denoted by $I$. The inner
product of $H$ or $K$ will be denoted by $\left \langle , \right \rangle$ if there is no confusion. For an operator $T\in B(H,K)$, let $T^{\ast}$ denote as usual its adjoint. A preserver problem generally deals with characterizing those maps on some specific algebraic structures which preserve a particular subset, property or relation. This subject has a long history and its origins goes back well over a century to the
so-called first linear preserver problem, due to Frobenius \cite {fro}, that determines linear
maps preserving the determinant of matrices. As we mentioned earlier, the main of this subject goal is to describe the general form of linear maps between two Banach algebras which preserve a certain property, or a certain class of elements,
or a certain relation. One of the most famous related problems is Kaplansky’s problem
\cite{kap} asking whether every surjective unital invertibility preserving linear map between two semisimple
Banach algebras is a Jordan homomorphism.
His question was motivated by two classical results, the result of Marcus and Moyls \cite{mar} on linear maps preserving eigenvalues of matrices and the Gleason-Kahane-Zelazko theorem \cite{gli,kah} stating that every unital invertibility preserving linear functional on a unital complex Banach algebra is necessarily multiplicative. The later this result was obtained independently by Gleason in \cite{gli} and Kahane-Zelazko in \cite{kah}, and was refined by Zelazko in \cite{zel}. In the non-commutative case,
the best known result so far are due to Sourour \cite{Sourour}. He answered to the
Kaplansky’s question in the affirmative for bijective unital
linear invertibility preserving maps acting on the algebra of all bounded operators on a Banach
space. Note that when the maps are unital, then preserving invertibility is equivalent to preserving spectrum. These results opened the gate for many authors who investigate linear (or additive) maps preserving spectrum; see for instance \cite{Aup,jaf,jaf2} and the references therein. Along this line, Molnar \cite{mol} investigated maps preserving the spectrum of operator
products without assuming linearity or additivity.\\

The local resolvent set, $\rho_{T}(x)$, of an operator $T \in B(H)$ at a point $x\in H$ is the union of all open subsets
$U$ of the complex plane $\mathbb{C}$ for which there is an analytic function $f : U\longrightarrow H$ such that
$(\mu I-T)f(\mu)=x$ for all $\mu \in U$. The complement of local resolvent
set is called the local spectrum of $T$ at $x$, denoted by
$\sigma_{T}(x)$, and is obviously a closed subset (possibly empty) of $\sigma(T)$, the spectrum
of $T$. We recall that an operator $T\in B(H)$ is said to have the single-valued extension
property (henceforth abbreviated to SVEP) if, for every open
subset $U$ of $\mathbb{C}$, there exists no nonzero analytic
solution, $f : U\longrightarrow H$, of the equation $$(\mu I-
T)f(\mu)=0, \quad \forall ~ \mu \in U.$$ Every operator $T \in
B(H)$ for which the interior of its point spectrum,
$\sigma_{p}(T)$, is empty enjoys this property.\\

 For every subset $F\subseteq \mathbb{C}$ the local spectral subspace $H_{T}(F)$ is defined by $$H_{T}(F)=\{x \in  H :\sigma_{T}(x) \subseteq F\}.$$ Clearly, if $F_{1}\subseteq F_{2}$ then $H_{T}(F_{1})\subseteq H_{T}(F_{2})$. For more information about these notions one may see the books \cite{Aiena,Laursen}.\\

The study of linear and nonlinear local spectra preserver problems attracted the attention of a number of authors. Bourhim and Ransford were the first ones to consider this type of preserver problem, characterizing in \cite {Bourhim} additive maps on $B(X)$, the algebra of all linear bounded operators on infinite-dimensional complex Banach space $X$, that preserves the local spectrum of
operators at each vector of $X$. Their results motivated several authors
to describe maps on matrices or operators that preserve local spectrum, local spectral
radius, and local inner spectral radius; see, for instance, the last section of the
survey article \cite{Bourhim2} and the references therein. Based on the results from the theory of linear preservers proved by Jafarian and Sourour \cite {jaf1}, Dolinar et al. \cite {Do}, characterised the form of maps preserving the lattice of sum of operators. They showed that the map (not necessarily linear) $\varphi: B(X) \rightarrow B(X)$ satisfies
Lat$(\varphi(T) + \varphi(S)) = $Lat$(T+S)$ for all $T, S \in B(X)$, if and only if there are a non zero
scalar $\alpha$ and a map $\phi: B(X) \rightarrow \mathbb{F}$ such that $\varphi(T)=\alpha T + \phi(T)I$ for all $T \in B(X)$ (See \cite[Theorem 1]{Do}), where $\mathbb{F}$ is the complex field $\mathbb{C}$ or the real field $\mathbb{R}$ and Lat$(T)$ is denoted the lattice of $T$, that is, the set of all invariant subspaces of $T$.
 They proved also, in the same paper, that a not necessarily linear maps $\varphi: B(X) \rightarrow B(X)$
satisfies Lat$(\varphi(T)\varphi(S))=$ Lat$(TS)$
(resp. Lat$(\varphi(T)\varphi(S)\varphi(T))=$ Lat$(TST)$, resp.
Lat$(\varphi(T)\varphi(S)+\varphi(S)\varphi(T))=$ Lat$(TS+ST)$) for all $T, S\in B(X)$, if and only if there is
a map $\phi: B(X) \rightarrow \mathbb{F}$ such that $\varphi(T)\neq 0$ if $T\neq 0$ and $\varphi(T)=\phi(T)T$ for all $T \in B(X)$ (See \cite[Theorem 2]{Do}).

For a Banach space $X$, it is well-known that $X_{T}(F)$, the local spectral subspace of $T$ associated with a subset $F$ of $\mathbb{C}$,
is an element of $Lat(T)$, so one can replace the lattice preserving property by the local
spectral subspace preserving property.
In \cite{elh}, the authors described additive maps on $B(X)$ that preserve the local spectral
subspace of operators associated with any singleton. More precisely, they proved that the
only additive map $\varphi$ on $B(X)$ for which
$X_{\varphi(T)}(\{\lambda\}) = X_{T}(\{\lambda\})$ for all $T \in B(X)$ and $\lambda \in \mathbb{C}$,
is the identity. In \cite{Benbouziane1}, Benbouziane et al. characterized the forms of surjective weakly
continuous maps $\varphi$ from $B(X)$ into $B(X)$ which satisfy
 $$X_{\varphi(T)-\varphi(S)}(\{\lambda\}) = X_{T-S}(\{\lambda\}), \quad  (T, S \in B(X),  \lambda \in \mathbb{C}).$$
Afterwards, in \cite{Benbouziane}, the authors studied surjective maps that preserve the local spectral subspace of the sum of two operators associated with non-fixed singletons. In other words,
they characterized surjective maps $\varphi$ on $B(X)$ which satisfy
$$X_{\varphi(T)+\varphi(S)}(\{\lambda\}) = X_{T+S}(\{\lambda\}), \quad  (T, S \in B(X),  \lambda \in \mathbb{C}).$$
They also gave a characterization of maps on $B(X)$ that preserve the local spectral subspace of the difference of operators associated with non-fixed singletons. Furthermore,
they investigated the product case as well as the triple product case. Namely, they described
surjective maps $\varphi$ on $B(X)$ satisfying
$$X_{\varphi(T)\varphi(S)}(\{\lambda\}) = X_{TS}(\{\lambda\}), \quad  (T, S \in B(X),  \lambda \in \mathbb{C}),$$
and also surjective maps $\varphi$ on $B(X)$ satisfying
$$X_{\varphi(T)\varphi(S)\varphi(T)}(\{\lambda\}) = X_{TST}(\{\lambda\}) \quad  (T, S \in B(X),   \lambda \in \mathbb{C}).$$
Bourhim and Lee \cite{Bb} investigated the form of all maps $\varphi_{1}$ and $\varphi_{2}$ on $B(X)$ such that, for every $T$ and $S$ in $B(X)$, the local spectra of $TS$ and $\varphi_{1}(T) \varphi_{2}(S)$ are the same at a nonzero fixed vector $x_{0}$. In this paper, We show that if $\varphi:B(H)\rightarrow B(H)$ is an additive map such that its range contains all operators of rank at most
two and satisfies
$$H_{\varphi(T)\varphi(S)^{\ast}}(\{\lambda\})= H_{TS^{\ast}}(\{\lambda\}),~~ (T, S \in B(H),   \lambda \in \mathbb{C}),$$
then there exist a unitary operator $V$ in $B(H)$ and a nonzero scalar $\mu$ such that $\varphi(T) = \mu TV^{\ast}$ for all $T \in B(H)$.
We also investigate the form of all maps $\varphi_{1}$ and $\varphi_{2}$ on $B(H)$ such that, for every $T$ and $S$ in $B(H)$, the local spectral subspaces of $TS^{\ast}$ and $\varphi_{1}(T) \varphi_{2}(S)^{\ast}$, associated with the singleton $\{\lambda\}$, coincide.

\section{ Preliminaries}

The first lemma summarizes some known basic and properties of the local
spectrum.

\begin{lemma} (See \cite{Aiena,Laursen}.)
Let $T \in B(H)$. For every $x,y \in H$
and a scalar $\alpha \in \mathbb{C}$ the following statements
hold.
\begin{itemize}
\item[(i)] $\sigma_{T}(\alpha x)=\sigma_{T}(x)$ if $\alpha \neq 0$, and $\sigma_{\alpha T}(x)=\alpha\sigma_{T}(x)$.
\item[(ii)] If $T x = \lambda x$ for some $\lambda \in \mathbb{C}$, then $\sigma_{T}(x) \subseteq
\{\lambda\} $. In particular, if $x\neq 0$ and $T$ has SVEP, then
$\sigma_{T}(x)=\{\lambda\}.$
\end{itemize}
\end{lemma}
In the next theorem we collect some of the basic properties of the subspaces
 $H_{T}(F)$.
\begin{lemma} (See \cite{Aiena,Laursen}.)
Let $T \in B(H)$. For $F\subseteq \mathbb{C}$ the following statements hold.\\
 (i) $H_{T}(F)$ is a T-hyperinvariant subspace of $H$.\\
 (ii) $(T-\lambda I)H_{T}(F)=H_{T}(F)$ for every $\lambda \in \mathbb{C}\backslash F$.\\
 (iii) If $x\in H$ satisfies $(T-\lambda I)x \in H_{T}(F)$, then $x\in H_{T}(F)$.\\
 (v) $ker(T-\lambda I) \subseteq H_{T}(F)$.\\
 (iv) $H_{\alpha T}(\lambda)=H_{T}(\frac{\lambda}{\alpha})$ for every $\lambda \in \mathbb{C}$ and non-zero scalar $\alpha$.\\
\end{lemma}

\noindent For a nonzero $h\in H$ and $T\in B(H)$, we use a useful notation defined by Bourhim and Mashreghi in \cite{Bourhim1}:

\begin{equation*}
\sigma^*_{T}(h):=\left \{
\begin{array}{ll}
\{0\}  & \hbox{ if \: $\sigma_{T}(h)=\{0\}$,}\\
\sigma_{T}(h)\setminus \{0\} & \hbox{ if \: $\sigma_{T}(h)\neq \{0\}$.}
\end{array}
\right.
\end{equation*}
For two nonzero vectors $x$ and $y$ in $H$, let $x\otimes y$ stands for the operator of rank at most one defined by $$(x \otimes y)z =\left \langle z,y \right \rangle x, \qquad \forall \ z \in H.$$
Note that every rank one operator in $B(H)$ can be written in this form,
and that every finite rank operator $T\in B(H)$ can be written as a finite
sum of rank one operators; i.e., $T=\sum_{i=1}^n x_{i}\otimes y_{i}$ for some $x_{i}, y_{i}\in H$ and $i=1,2,...,n$. By $F(H)$ and $F_{n}(H)$, we mean the set of all finite rank operators in $B(H)$. and the set of all operators of rank at most $n$, $n$ is a positive integer, respectively.\\

The following lemma is an elementary observation which discribes the nonzero
local spectrum of any rank one operator.

\begin{lemma} (See \cite[Lemma 2.2]{Bourhim1}.)
Let $h_{0}$ be a nonzero vector in $H$. For every vectors $x,y\in H$, we have
\begin{equation*}
\sigma^{*}_{x\otimes y}(h_{0}):=\left \{
\begin{array}{ll}
\{0\}  & \hbox{ if \: $\left \langle h_{0},y \right \rangle=0$,}\\
\left \langle x,y \right \rangle & \hbox{ if \: $\left \langle h_{0},y \right \rangle\neq 0$.}
\end{array}
\right.
\end{equation*}

\end{lemma}

The following theorem, which may be of independent interest, gives a spectral characterization of rank one
operators in term of local spectrum.

\begin{theorem} (See \cite[Theorem 4.1]{Bourhim1}.)
For a nonzero vector $h\in H$ and
a nonzero operator $R\in B(H),$ the following statements are equivalent.\\
(a) $R$ has rank one.\\
(b) $\sigma^{*}_{RT}(h)$ contains at most one element for all $T\in B(H)$.\\
(c) $\sigma^{*}_{RT}(h)$ contains at most one element for all $T\in F_{2}(H)$.
\end{theorem}

The following Lemma is a key tool for the proofs in the sequel.

 \begin{lemma} (See \cite[Lemma 1.6]{Benbouziane}.)
 Let $h$ be a nonzero vector in $H$ and $T, S \in B(H)$. If $H_{T}(\{\lambda\})=H_{S}(\{\lambda\})$ for all $\lambda \in \mathbb{C}$. Then,
$\sigma_{T}(h)=\{\mu\}$ if and only if $\sigma_{S}(h)=\{\mu\}$ for
all $\mu \in \mathbb{C}$.
\end{lemma}

Moreover, this theorem will be useful in the proofs of our main result.

\begin{theorem} (See \cite[Theorem 2.1]{Benbouziane}.)
 Let $T, S \in B(H)$. The following statements are equivalent.\\
$(1)$ $T = S$.\\
$(2)$ $H_{TR}(\{\lambda\}) = H_{SR}(\{\lambda\})$ for all
$\lambda \in \mathbb{C}$ and $R\in F_{1}(H)$.
\end{theorem}

The next theorem describes additive maps on $B(H)$ that preserve the local spectral
subspace of operators associated with any singleton.

\begin{theorem} (See \cite[Theorem 2.1]{elh}.)
Let $\varphi: B(H) \rightarrow B(H)$ be an additive map such that
$H_{\varphi(T)}(\{\lambda\})=H_{T}(\{\lambda\})$ for all $T \in B(H)$ and $\lambda \in \mathbb{C}$.
Then $\varphi(T)=T$ for all $T \in B(H)$.
\end{theorem}

The following theorem will be useful in the sequel. We recall
that if $h:\mathbb{C}\rightarrow \mathbb{C}$ is a ring homomorphism, then an additive map $A: H \rightarrow H$ satisfying
$A(\alpha x)=h(\alpha)x ,(x \in H, \alpha \in \mathbb{C})$ is called an $h$-quasilinear operator.

\begin{theorem} (See \cite[Theorem 3.3]{om}.)
  Let $\varphi: F(H) \rightarrow F(H)$ be a bijective additive map preserving rank one operators in both directions. Then there exist a ring automorphism $h:\mathbb{C}\rightarrow \mathbb{C}$, and either there are $h$-quasilinear bijective maps $A: H \rightarrow H$ and $B:H \rightarrow H$ such that
$$\varphi(x \otimes y) = Ax \otimes By, ~~~ x, y \in H,$$
or there are $h$-quasilinear bijective maps $C:H\rightarrow H$ and $D:H \rightarrow H$ such
that
$$\varphi(x \otimes y) = Cy \otimes Dx,~~  x,y \in H.$$
\end{theorem}

Note that, if in Theorem 2.8 the map $\varphi$ is linear, then $h$ is the identity map on $\mathbb{C}$ and so
the maps $A, B, C$ and $D$ are linear.

\section{ \normalsize\bf  Main Results}
The following theorem is the first main result of this paper, which characterizes those maps preserving the local spectral subspace of skew-product operators.

\begin{theorem}\label{maintheorem}
Let $\varphi:B(H)\rightarrow B(H)$ be an additive map such that its range contains $F_{2}(H)$. If
\begin{equation}\label{t1}
H_{\varphi(T)\varphi(S)^{\ast}}(\{\lambda\})= H_{TS^{\ast}}(\{\lambda\}),~~ (T, S \in B(H),   \lambda \in \mathbb{C}),
\end{equation}

then there exist a unitary operator $V$ in $B(H)$ and a nonzero scalar $\mu$ such that $\varphi(T) = \mu TV^{\ast}$ for all $T \in B(H)$.
\end{theorem}

\begin{proof}
The proof breaks down into several claims.\\

{\bf Claim 1.} $\varphi$ is injective.\\

If $\varphi(T)=\varphi(S)$ for some $T, S \in B(H)$, we get that
$$H_{TR^{\ast}}(\{\lambda\})=H_{\varphi(T)\varphi(R)^{\ast}}(\{\lambda\}) =H_{\varphi(S)\varphi(R^{\ast})}(\{\lambda\})= H_{SR^{\ast}}(\{\lambda\})$$
for all $R\in F_{1}(X)$ and $\lambda \in \mathbb{C}$. By Theorem 2.6, we see that $T=S$ and hence $\varphi$ is injective.\\

{\bf Claim 2.} $\varphi$ preserves rank one operators in both directions.\\

Let $R=x \otimes y$ be a rank one operator where $x,y \in H$. Note that, $\varphi(R)\neq0$, since $\varphi(0)=0$ and $\varphi$ is injective. Let $T \in B(H)$ be an arbitrary operator. Since $RT^{\ast}x=\left \langle x,Ty \right \rangle x$ and $RT^{\ast}$ has the SVEP, then $\sigma_{RT^{\ast}}(x)=\{\left \langle x,Ty \right \rangle \}$. We have $$x \in H_{RT^{\ast}}(\{\left \langle x,Ty \right \rangle\})=H_{\varphi(R)\varphi(T)^{\ast}}(\{\left \langle x,Ty \right \rangle\}).$$ As the range of $\varphi$ contains $F_{2}(H)$, using Lemma 2.5, $\sigma^{*}_{\varphi(R)S^{\ast}}(x)$ contains at most one element for all operators $S\in F_{2}(H)$. By Theorem 2.4, we see that $\varphi(R)$ has rank one. The converse holds in a similar way and thus $\varphi$ preserves the rank
one operators in both directions.\\

{\bf Claim 3.} $\varphi$ is linear.\\

We show that $\varphi$ is homogeneous. Let $R$ be an arbitrary rank-one operator. By the previous claim, there exists a rank one operator $S$ in $B(H)$ such that $\varphi(S)=R$. For every $\alpha, \lambda \in \mathbb{C}$ with $\alpha\neq0$ and $T \in B(H)$, we have
\begin{align*}
H_{\alpha \varphi(T)R^{\ast}}(\{\lambda\})&=H_{\alpha \varphi(T)\varphi(S)^{\ast}}(\{\lambda\})=H_{\varphi(T)\varphi(S)^{\ast}}(\{\frac{\lambda}{\alpha}\})\nonumber\\
&= H_{TS^{\ast}}(\{\frac{\lambda}{\alpha}\})= H_{(\alpha T)S^{\ast}}(\{\lambda\})\nonumber\\
&=H_{\varphi(\alpha T)\varphi(S)^{\ast}}(\{\lambda\})=H_{\varphi(\alpha T)R^{\ast}}(\{\lambda\}).
\end{align*}

By Theorem 2.6, we see that $\varphi(\alpha T)=\alpha \varphi(T)$.
 Since $\varphi$ is assumed to be additive, the map $\varphi$ is, in fact, linear.\\

{\bf Claim 4.} There are bijective linear mappings $A:H \rightarrow H$ and $B :H \rightarrow H$
such that $\varphi(x \otimes y) = Ax \otimes By$ for all $x, y \in H$.\\

By the previous claim $\varphi$ is a bijective linear map from $F(H)$ onto $F(H)$ and preserves
rank one operators in both directions, thus by Theorem 2.8, either there are bijective linear
mappings $A: H \rightarrow H$ and $B: H \rightarrow H$ such that
\begin{equation}\label{e1}
\varphi(x \otimes y) = Ax \otimes By, ~~~ x, y \in H,
\end{equation}

or there are bijective linear mappings $C:H \rightarrow H$ and $D:H \rightarrow H$ such
that
\begin{equation}\label{e2}
  \varphi(x \otimes y) = Cy \otimes Dx,~~  x, y \in H.
\end{equation}

Assume that $\varphi$ takes the form $(3)$. Let $y_{1}$ be a nonzero vector in $H$.  Choose a nonzero vector $v$ such that $\left \langle y_{1},v \right \rangle=0$. Set $x=C^{-1}y_{1}$, since $x$ and $v$ are nonzero vectors in $H$, there exists a $y \in H$ such that $\left \langle x,y \right \rangle=1$ and $\left \langle v,y \right \rangle \neq 0$, since $x \otimes y$ is idempotent, we have
\begin{align*}
H_{x \otimes y}(\{\lambda\})&=H_{(x \otimes y)(x \otimes y)}(\{\lambda\})\nonumber\\
&=H_{(x \otimes y)(y \otimes x)^{\ast}}(\{\lambda\})\nonumber\\
&=H_{(Cy \otimes Dx)(Cx \otimes Dy)^{\ast}}(\{\lambda\})\nonumber\\
&=H_{(Cy \otimes Dx)(Dy \otimes Cx)}(\{\lambda\})\nonumber\\
&=H_{\left \langle Dy,Dx \right \rangle(Cy \otimes Cx)}(\{\lambda\})\nonumber\\
&=H_{\left \langle Dy,Dx \right \rangle(Cy \otimes y_{1})}(\{\lambda\}).
 \end{align*}
On the other hand, since $\left \langle y_{1},v \right \rangle=0$, we have $\sigma^{*}_{Cy \otimes y_{1}}(v)=\{0\}$ and consequently $\left \langle Dy,Dx \right \rangle \sigma_{Cy \otimes y_{1}}(v)=\{0\}$. This implies that
  $$v \in H_{\left \langle Dy,Dx \right \rangle Cy \otimes y_{1}}(\{0\})=H_{x \otimes y}(\{0\}).$$
 Using Lemma 2.5, $\sigma_{x \otimes y}(v)=\{0\}$. But lemma 2.3 implies that $$\sigma^{*}_{x \otimes y}(v)\neq\{0\}.$$
This contradiction shows that $\varphi$ only takes the form $(2)$.\\

{\bf Claim 5.} $A$ and $B$ are bounded unitary operators multiplied by positive scalars $\alpha$ and $\beta$ such that $\alpha\beta=1$.\\

Let $x,y$ be nonzero vectors in $H$, since $\sigma_{(x \otimes y)(y \otimes x)}(x)=\{\left \langle y,y \right \rangle \left \langle x,x \right \rangle\}$,
 by the previous claim, we have
\begin{align*}
H_{(x \otimes y)(y\otimes x)}(\{\|y\|^{2} \|x\|^{2}\})&=H_{(x \otimes y)(x \otimes y)^{\ast}}(\{\|y\|^{2} \|x\|^{2}\})\nonumber\\
&=H_{(Ax \otimes By)(Ax \otimes By)^{\ast}}(\{\|y\|^{2} \|x\|^{2}\})\nonumber\\
&=H_{(Ax \otimes By)(By \otimes Ax)}(\{\|y\|^{2} \|x\|^{2}\})\nonumber\\
&=H_{\left \langle By,By \right \rangle(Ax \otimes Ax)}(\{\|y\|^{2} \|x\|^{2}\}).
\end{align*}
By the Lemma 2.5, we see that\\

\begin{align}\label{e0}
\{\|y\|^{2} \|x\|^{2}\}&=\sigma_{(x \otimes y)(y\otimes x)}(x)\nonumber\\
&=\sigma_{(\left \langle By,By \right \rangle(Ax \otimes Ax))}(x)
=\{\|By\|^{2} \|Ax\|^{2}\}.
\end{align}

Now, let $y_{0}$ be a fixed unit vector in $H$ and let $\alpha=\frac{1}{\|By_{0}\|}$. By $(4)$, we have
$$\|Ax\|^{2}=\alpha^{2}\|x\|^{2}$$
for all $x\in H$. Hence, $U=\frac{1}{\alpha}A$ is an isometry and thus it is a unitary operator in $B(H)$, because $A$ is bijective. Similarly, fix a unit vector $x_{0} \in H$ and take $\beta=\frac{1}{\|Ax_{0}\|}$, and note that $V=\frac{1}{\beta}B$ is a unitary operator in $B(H)$. Finally, by $(4)$, we see that $\alpha\beta=1$.\\

{\bf Claim 6.} $A^{\ast}$ and $I$ are linearly dependent.\\

Assume, by the way of contradiction, that there exists a nonzero vector $x\in H$ such that $Ax$ and $x$ are linearly independent. Let $u\in H$ be a vector such that $\left \langle x,u \right \rangle=1$ and $\left \langle A^{\ast}x,u \right \rangle=0$. Since $\sigma_{x \otimes u}(x)=\{1\}$, then
\begin{align*}
x\in H_{x \otimes u}(\{1\})&=X_{(x \otimes u)(x \otimes u)}(\{1\})\nonumber\\
&=H_{(x \otimes u)(u \otimes x)^{\ast}}(\{1\})\nonumber\\
&=H_{\varphi(x \otimes u)\varphi(u \otimes x)^{\ast}}(\{1\})\nonumber\\
&=H_{(Ax \otimes Bu)(Au \otimes Bx)^{\ast}}(\{1\})\nonumber\\
&=H_{(Ax \otimes Bu)(Bx \otimes Au)}(\{1\})\nonumber\\
&=H_{\left \langle Bx,Bu \right \rangle(Ax \otimes Au)}(\{1\}),
\end{align*}
using Lemma 2.5, we have
\begin{align*}
\{1\}=\sigma_{x \otimes u}(x)=\sigma_{\left \langle Bx,Bu \right \rangle(Ax \otimes Au)}(x)=\{0\}.
\end{align*}
This contradiction shows that there is a nonzero scalar $\gamma \in \mathbb{C}$ such that
$A^{\ast}=\gamma I$.\\

{\bf Claim 7.} $\varphi(T)=\mu TV^{\ast}$ for all $T \in B(H)$, where $V$ is unitary operators and $\mu$ is a nonzero scalar.\\

By claim 5 we shall assume that $A=U$ and $B=V$ for some unitary operators $U, V\in B(H)$. Using the previous claim and $(1)$, for every rank one operator $R\in B(H)$ and every operator $T\in B(H)$ we have
\begin{align*}
H_{\varphi(T)\varphi(R)^{\ast}}(\{\lambda\})&=H_{TR^{\ast}}(\{\lambda\})\nonumber\\
&=H_{UTR^{\ast}U^{\ast}}(\{\lambda\})\nonumber\\
&=H_{UTV^{\ast}VR^{\ast}U^{\ast}}(\{\lambda\})\nonumber\\
&=H_{UTV^{\ast}(URV^{\ast})^{\ast}}(\{\lambda\})\nonumber\\
&=H_{UTV^{\ast}\varphi(R)^{\ast}}(\{\lambda\}).
\end{align*}
Since $\varphi$ preserves rank one operators in both directions, Theorem 2.6 shows that $\varphi(T)=UTV^{\ast}$ for all $T\in B(H)$. Claim 6 tells us that for ever $T\in B(H)$ we have $\varphi(T)=\mu TV^{\ast}$ for some $\mu \in \mathbb{C}$.
\end{proof}

From this result, it is easy to deduce a generalization for the case of two different Hilbert spaces $H, K$.

\begin{corollary}
Suppose $U\in B(H,K)$ be a unitary operator. Let $\varphi$ be an additive map from $B(H)$ into $B(K)$ such that its range contains $F_{2}(K)$. If
$$K_{\varphi(T)\varphi(S)^{\ast}}(\{\lambda\})=UH_{TS^{\ast}}(\{\lambda\}),~~ (T, S \in B(H),   \lambda \in \mathbb{C}).$$
Then there exist a unitary operator $V:H \rightarrow K$ and a nonzero scalar $\mu$ such that $\varphi(T)=\mu UTV^{\ast}$ for all $T \in B(H)$.
\end{corollary}

\begin{proof}
We consider the map $\psi:B(H) \rightarrow B(H)$ defined by $\psi(T)=U^{\ast}\varphi(T)U$ for all $T\in B(H)$.
We have,
\begin{align*}
H_{\psi(T)\psi(S)^{\ast}}(\{\lambda\})&=H_{U^{\ast}\varphi(T)UU^{\ast}\varphi(S)^{\ast}U}(\{\lambda\})\nonumber\\
&=H_{U^{\ast}\varphi(T)\varphi(S)^{\ast}U}(\{\lambda\})\nonumber\\
&=U^{\ast} K_{\varphi(T)\varphi^{\ast}(S)}(\{\lambda\})\nonumber\\
&= H_{TS^{\ast}}(\{\lambda\})
\end{align*}
for every $T, S \in B(H)$ and $\lambda \in \mathbb{C}$. So by Theorem 3.1, there exist a unitary operator $P:H \rightarrow H$ and $\mu \in \mathbb{C}$ such that $\psi(T)=\mu TP^{\ast}$ for all $T \in B(H)$. Therefore $\varphi(T)=\mu UTV^{\ast}$ for all $T \in B(H)$, where $V=UP$.
\end{proof}

In the next theorem, we investigate the form of all maps $\varphi_{1}$ and $\varphi_{2}$ on $B(H)$ such that, for every $T$ and $S$ in $B(H)$, the local spectral subspaces of $TS^{\ast}$ and $\varphi_{1}(T) \varphi_{2}(S)^{\ast}$, associated with the singleton $\{\lambda\}$, coincide.

\begin{theorem}\label{maintheorem}
Let $\varphi_{1}$ and $\varphi_{2}$ be additive maps from $B(H)$ into $B(H)$ which satisfy
\begin{equation}\label{t2}
H_{\varphi_{1}(T)\varphi_{2}(S)^{\ast}}(\{\lambda\})=H_{TS^{\ast}}(\{\lambda\}),~~ (T, S \in B(H),   \lambda \in \mathbb{C}).
\end{equation}
If the range of $\varphi_{1}$ and $\varphi_{2}$ contain $F_{2}(H)$, then $\varphi_{2}(I)^{\ast}$ is invertible, and $\varphi_{1}(T) = T(\varphi_{2}(I)^{\ast})^{-1}$ and $\varphi_{2}(T) =\varphi_{2}(I)^{\ast}T$ for all $T \in B(H)$.
 \end{theorem}

\begin{proof}
The proof is rather long and we break it into several claims.\\

{\bf Claim 1.} $\varphi_{1}$ is a one to one linear map preserving rank one
operators in both directions.\\

Similar to the proof of Theorem 3.1, we can shows that $\varphi_{1}$ is a one to one linear map preserving rank one
operators in both directions.\\

{\bf Claim 2.} There are bijective linear mappings $A:H \rightarrow H$ and $B : H \rightarrow H$
such that $\varphi_{1}(x \otimes y) = Ax \otimes By$ for all $x, y \in H$.\\

By the claim 1 $\varphi_{1}:F(H) \rightarrow F(H)$ is a bijective linear map which preserves
rank one operators in both directions. Thus by Theorem 2.8, $\varphi_{1}$ has one of the following forms.\\
$(1)$ There exist bijective linear
maps $A: H \rightarrow H$ and $B:H \rightarrow H$ such that
\begin{equation}\label{e3}
\varphi_{1}(x \otimes y) = Ax \otimes By, ~~~ x, y \in H.
\end{equation}

$(2)$ There exist bijective linear maps $C:H \rightarrow H$ and $D:H \rightarrow H$ such
that
\begin{equation}\label{e4}
\varphi_{1}(x \otimes y) = Cy \otimes Dx,~~  x, y \in H.
\end{equation}

Assume that $\varphi_{1}$ takes the form $(7)$. Let $y$ be a nonzero vector in $H$, choose a nonzero vector $v\in H$
such that $\left \langle \varphi_{2}(I)^{\ast}y,v \right \rangle=0$. Set $x=D^{-1}v$, since $x$ and $y$ are nonzero
vectors in $H$, there exists a vector $u \in H$ such that $\left \langle y,u \right \rangle \neq0$ and $\left \langle x,u \right \rangle \neq0$. Since $\sigma^{*}_{(Cu \otimes v) \varphi_{2}(I)^{\ast}}(y)=\{0\}$, we have
$$y \in H_{(Cu \otimes v) \varphi_{2}(I)^{\ast}}(\{0\})=H_{(Cu \otimes Dx) \varphi_{2}(I)^{\ast}}(\{0\})=H_{x \otimes u}(\{0\}).$$
 Using Lemma 2.5, $\sigma_{x \otimes u}(y)=\{0\}$. But lemma 2.3 implies that $$\sigma^{*}_{x \otimes u}(y)=\{\left \langle x,u \right \rangle \}\neq\{0\}.$$
This contradiction shows that $\varphi_{1}$ only takes the form $(6)$.\\

{\bf Claim 3.} For every $x, y\in H$, $\left \langle x,y \right \rangle = \left \langle Ax,\varphi_{2}(I)(By) \right \rangle$.\\

Assume that $x$ and $ y$ are arbitrary vectors in $H$. We have, $\sigma_{x \otimes y}(x)=\{\left \langle x,y \right \rangle\}$, so the previous claim and $(5)$ imply that
$$x\in H_{x \otimes y}(\{\left \langle x,y \right \rangle\})=H_{\varphi_{1}(x \otimes y)\varphi_{2}(I)^{\ast}}(\{\left \langle x,y \right \rangle\})=H_{(Ax \otimes By) \varphi_{2}(I)^{\ast}}(\{\left \langle x,y \right \rangle\}).$$
 Assume first that $\left \langle x,y \right \rangle \neq 0$, using lemma 2.5,
 $$\{0\} \neq \{\left \langle x,y \right \rangle\}=\sigma_{x \otimes y}(x)=\sigma_{\varphi_{1}(x\otimes y)\varphi_{2}(I)^{\ast}}(x)=\sigma_{(Ax \otimes By)\varphi_{2}(I)^{\ast}}(x),$$
 which means that $\left \langle x,\varphi_{2}(I)(By) \right \rangle \neq 0$. Then Lemma 2.3 implies that
$$\{\left \langle x,y \right \rangle\}=\sigma_{x \otimes y}(x)=\sigma_{(Ax \otimes By)\varphi_{2}(I)^{\ast}}(x)=\{\left \langle Ax,\varphi_{2}(I)(By) \right \rangle\}.$$

Now, if $\left \langle x,y \right \rangle=0$, we choose a vector $u \in H$ such that $\left \langle x,u \right \rangle \neq 0$. By application of what has
been shown previously to both $u$ and $x+u$, we have $\left \langle x,u \right \rangle=\left \langle Ax,\varphi_{2}(I)(By) \right \rangle$ and $\left \langle x,y+u \right \rangle=\left \langle Ax,\varphi_{2}(I)(B(u+y)) \right \rangle$. So
\begin{align*}
\left \langle x,y \right \rangle+\left \langle x,u \right \rangle &=\left \langle x,y+u \right \rangle \nonumber\\
&=\left \langle Ax,\varphi_{2}(I)(B(y+u) \right \rangle \nonumber\\
&=\left \langle Ax,\varphi_{2}(I)(By) \right \rangle+\left \langle Ax,\varphi_{2}(I)(Bu) \right \rangle \nonumber\\
&=\left \langle Ax,\varphi_{2}(I)(By) \right \rangle+\left \langle x,u \right \rangle.
\end{align*}
This shows that $\left \langle x,y \right \rangle=\left \langle Ax,\varphi_{2}(I)(By) \right \rangle$ in this case too.\\

{\bf Claim 4.} $\varphi_{2}(I)^{\ast}$ is invertible.\\

It is clear that $\varphi_{2}(I)^{\ast}$ is injective, if not, there is a nonzero vector $y\in H$ such that $\varphi_{2}(I)^{\ast}y=0$. Take $x=A^{-1}y$, and let $u \in H$ be a vector such that $\left \langle x,u \right \rangle=1$. By the previous claim, we have
$1=\left \langle x,u \right \rangle =\left \langle Ax,\varphi_{2}(I)(Bu) \right \rangle=\left \langle y,\varphi_{2}(I)(Bu) \right \rangle=\left \langle \varphi_{2}(I)^{\ast}y,Bu \right \rangle=0$. This contradiction tells us that $\varphi_{2}(I)^{\ast}$ is injective.
Now, we show that $A$ is continuous and $(\varphi_{2}(I))B=(A^{\ast})^{-1}$. Assume that $(x_{n})_{n}$
is a sequence in $H$ such that $\lim_{n\longrightarrow \infty} x_{n}=x\in H$ and $\lim_{n\longrightarrow \infty} Ax_{n}=y\in H$. Then, for every $u\in H$, we have
\begin{align*}
\left \langle y,\varphi_{2}(I)(Bu) \right \rangle&=\lim_{n\longrightarrow \infty} \left \langle Ax_{n},\varphi_{2}(I)(Bu) \right \rangle \nonumber\\
&=\lim_{n\longrightarrow \infty} \left \langle x_{n},u) \right \rangle =\left \langle x,u \right \rangle=\left \langle Ax,\varphi_{2}(I)(By) \right \rangle.
\end{align*}
 $$$$
 Since $B$ is bijective and $u\in H$ is an arbitrary vector, the closed graph theorem shows
that $A$ is continuous. Moreover, we have $\left \langle x,y \right \rangle=\left \langle Ax,\varphi_{2}(I)(By)\right \rangle=\left \langle x,A^{\ast}\varphi_{2}(I)(By) \right \rangle$
for all $x,y \in H$, and thus $I=A^{\ast}\varphi_{2}(I)B$.
It follows that $\varphi_{2}(I)^{\ast}$ is invertible.\\

{\bf Claim 5.} $(A^{\ast})^{-1}$ and $I$ are linearly dependent.\\

Assume, by the way of contradiction, that there exists a nonzero vector $x\in H$ such that $Ax$ and $x$ are linearly independent.
Let $y$ be a vector in $H$ such that $\left \langle x,y \right \rangle=1$ and $\left \langle (A^{\ast})^{-1}x,y \right \rangle=0$. We have
$(x \otimes y)x=x$ and $A(x \otimes y)(A^{\ast})^{-1}x=0$, then $\sigma_{x \otimes y}(x)=\{1\}$ and $\sigma_{A(x \otimes y)(A^{\ast})^{-1}}(x)=\{0\}$. Since
$H_{x \otimes y}(\{\lambda\}) =H_{\varphi_{1}(x \otimes y)\varphi_{2}(I)^{\ast}}(\{\lambda\})$, using Lemma 2.5 and claim 2, we have
\begin{align*}
\{1\}&=\sigma_{x \otimes y}(x)\nonumber\\
&=\sigma_{\varphi_{1}(x \otimes y)\varphi_{2}(I)^{\ast}}(x)\nonumber\\
&=\sigma_{(Ax \otimes By)\varphi_{2}(I)^{\ast}}(x)\nonumber\\
&=\sigma_{(Ax \otimes \varphi_{2}(I)By)}(x)\nonumber\\
&=\sigma_{A(x \otimes y)(A^{\ast})^{-1}}(x)=\{0\}.
\end{align*}
This contradiction shows that there is a nonzero scalar $\alpha \in \mathbb{C}$ such that
$(A^{\ast})^{-1}=\alpha I$.\\

{\bf Claim 6.} $\varphi_{1}$ and $\varphi_{1}$ have the desired forms.\\

We define the map $\psi_{1}: B(H) \rightarrow B(H)$ by $\psi_{1}(T)=\varphi_{1}(T)\varphi_{2}(I)^{\ast}$ for all $T\in B(H)$.
We have,
$$H_{\psi_{1}(T)}(\{\lambda\}) =H_{\varphi_{1}(T)\varphi_{2}(I)^{\ast}}(\{\lambda\})=H_{T}(\{\lambda\})$$
 for all $T \in B(H)$ and $\lambda \in \mathbb{C}$. So by Theorem 2.7, $\psi_{1}(T)=T$ for all $T \in B(H)$, therefore $\varphi_{1}(T)=T(\varphi_{2}(I)^{\ast})^{-1}$ for all $T \in B(H)$. Once again, we consider the map $\psi_{2}:B(H) \rightarrow B(H)$ defined by $\psi_{2}(T)=\varphi_{1}(I)\varphi_{2}(T^{\ast})^{\ast}$ for all $T\in B(H)$.
We see that for all $T \in B(H)$ and $\lambda \in \mathbb{C}$,
$$H_{\psi_{2}(T)}(\{\lambda\})=H_{T}(\{\lambda\}),$$ by Theorem 2.7, $\psi_{2}(T)=T$ for all $T \in B(H)$. Hence $\varphi_{2}(T)=(\varphi_{1}(I))^{-1}T$ for all $T \in B(H)$.
\end{proof}

Theorem 3.3 leads directly to the following corollary.
\begin{corollary}
Suppose $U\in B(H,K)$ be a unitary operator. Let $\varphi_{1}$  and $\varphi_{2}$ be two additive map from $B(H)$ into $B(K)$ which satisfy
$$K_{\varphi_{1}(T)\varphi_{2}(S)^{\ast}}(\{\lambda\})=UH_{TS^{\ast}}(\{\lambda\}),~~  (T, S \in B(H),   \lambda \in \mathbb{C}).$$
If the range of $\varphi_{1}$  and $\varphi_{2}$ contain $F_{2}(K)$, then there exists a bijective linear map $V:K\rightarrow H$ such that $\varphi_{1}(T) = UTV$ and $\varphi_{2}(T) =V^{-1}TU^{\ast}$ for all $T \in B(H)$.
\end{corollary}

\begin{proof}
We consider the maps $\psi_{1}:B(H) \rightarrow B(H)$ defined by $\psi_{1}(T)=U^{\ast}\varphi_{1}(T)U$ and $\psi_{2} : B(H) \rightarrow B(H)$ defined by $\psi_{2}(T)=U^{\ast}\varphi_{2}(T)U$ for all $T\in B(H)$.
We have,
\begin{align*}
H_{\psi_{1}(T)\psi_{2}(S)^{\ast}}(\{\lambda\})& =H_{U^{\ast}\varphi_{1}(T)UU^{\ast}\varphi_{2}(S)^{\ast}U}(\{\lambda\})\nonumber\\
&=H_{U^{\ast}\varphi_{1}(T)\varphi_{2}(S)^{\ast}U}(\{\lambda\})\nonumber\\
&=U^{-1} K_{\varphi_{1}(T)\varphi_{2}(S)^{\ast}}(\{\lambda\})
= H_{TS^{\ast}}(\{\lambda\})
\end{align*}
for every $T, S \in B(H)$ and $\lambda \in \mathbb{C}$. So by Theorem 3.3, $\psi_{1}(T)=TP^{-1}$ and $\psi_{2}(T)=PT$ for all $T \in B(H)$, where $P=\psi_{2}(I)^{\ast}$. Therefore $\varphi_{1}(T) = UTV$ and $\varphi_{2}(T)=V^{-1}TU^{-1}$ for all $T \in B(H)$, where $V=(UP)^{-1}$.
\end{proof}

\end{document}